\documentclass[1pt]{article}

\usepackage{amsmath, amssymb, amsthm, amsfonts, cases}
\usepackage{mathrsfs}
\usepackage{url}
\usepackage{authblk}
\usepackage[usenames]{color}
\usepackage{graphicx}
\usepackage{float}
\usepackage{subfigure}
\usepackage{caption2}
\usepackage{geometry}
\usepackage{multirow}
\allowdisplaybreaks

\newlength{\defbaselineskip}
\newcommand{\setlinespacing}[1]%
           {\setlength{\baselineskip}{#1 \defbaselineskip}}

\theoremstyle{plain}
\newtheorem{theorem}{Theorem}[section]

\newtheorem{problem}[theorem]{Problem}
\newtheorem{lemma}[theorem]{Lemma}

\theoremstyle{definition}
\newtheorem{definition}{Definition}[section]
\newtheorem{assumption}{Assumption}[section]

\makeatletter\@addtoreset{equation}{section} \makeatother

\begin{document}

\title{On the existence of optimal controls for backward stochastic
partial differential equations\thanks{This work was supported by the Natural Science Foundation of Zhejiang Province
for Distinguished Young Scholar  (No.LR15A010001),  and the National Natural
Science Foundation of China (No.11471079, 11301177).}}

\date{}
\author[a]{Qingxin Meng}
\author[b]{Yang Shen\footnote{Corresponding author.
\authorcr
\indent E-mail address: mqx@hutc.zj.cn(Q. Meng); yangshen@yorku.ca, skyshen87@gmail.com(Y. Shen);
peng.shi@adelaide.edu.au(P. Shi)}}
\author[c,d]{Peng Shi}

\affil[a]{\small{Department of Mathematical Sciences, Huzhou University, Zhejiang 313000, China}}
\affil[b]{\small{Department of Mathematics and Statistics, York University, Toronto, Ontario, M3J 1P3, Canada}}
\affil[c]{School of Electrical and Electronic Engineering, The University of Adelaide, Adelaide, SA 5005, Australia}
\affil[d]{College of Engineering and Science, Victoria University, Melbourne, VIC 8001, Australia}

\maketitle

\begin{abstract}
\noindent
This paper is concerned with the existence of optimal controls for backward stochastic
partial differential equations with random coefficients, in which the control systems
are represented in an abstract evolution
form, i.e. backward stochastic evolution equations. Under some growth and monotonicity conditions
on the coefficients and suitable assumptions on the Hamiltonian function, the existence of the optimal
control boils down to proving the uniqueness and existence of a solution to the stochastic Hamiltonian system,
i.e. a fully coupled forward-backward stochastic evolution equation. Using some a prior estimates, we prove
the uniqueness and existence via the method of continuation. Two examples of linear-quadratic control
are solved to demonstrate our results.
\end{abstract}

\section{Introduction}

In this paper, we consider an optimal control problem
under a stochastic backward system in infinite dimensions. More specifically, the control
system is given by a backward stochastic partial differential equation
in the abstract evolution form:
\begin{eqnarray}\label{eq:1.1}
\left\{
\begin{aligned}
dy(t)=& \ [A(t)y(t)+B(t)z(t)+D(t)u(t)+G(t)]dt+z(t)dW(t),\\
y(T)=& \ \xi ,
\end{aligned}
\right.
\end{eqnarray}
with the cost functional:
\begin{eqnarray} \label{eq:1.2}
J(u(\cdot))= \mbox{E}\bigg[\int_0^Tl(t,y(t),z(t),u(t))dt+h(y(0))\bigg],
\end{eqnarray}
where $A$ is a given stochastic evolution operator, $B, D,G, \xi, l$ and $h$ are given
random maps, and $W (\cdot)$ is a one-dimensional standard Brownian motion.
The state process $(y(\cdot), z(\cdot))$ and the control process
$u(\cdot)$ take values in Hilbert spaces $V \times H$ and $U$, respectively.
The objective of the optimal control problem is to find a control process that minimizes
the cost functional \eqref{eq:1.2} over the set of admissible controls.
The work in \cite{meng2013stochastic} established necessary
and sufficient maximum principles for a more general backward control system of infinite dimensions.
However, the existence of optimal controls was not discussed thoroughly.
This paper attempts to fill in the gap in \cite{meng2013stochastic}, and
establish the existence conditions of an optimal control under system \eqref{eq:1.1}-\eqref{eq:1.2}.

The existence of optimal controls for various control systems is a fundamental problem in stochastic optimal control theory
which has attracted comprehensive attention in the past years.
One approach to study the existence of optimal controls is based on the dynamic programming principle
and the solvability of the corresponding HJB equation in a sufficiently regular sense.
The work in \cite{karoui1987compactification} used the compactness argument and proved the existence of
an optimal Markovian relaxed control for systems with degenerate diffusions.
Based on an approximation of stochastic control systems with smooth coefficients,
the existence of optimal controls for stochastic control systems was investigated in \cite{buckdahn2010existence}
with the cost functional given by a controlled backward stochastic differential equation.
Also some earlier works along this research line can be found in \cite{bismut1976theorie} and \cite{davis1973on}, and the references therein. However, since all the coefficients in system \eqref{eq:1.1}-\eqref{eq:1.2} are random,
the corresponding HJB equation becomes a nonlinear backward stochastic partial differential equation,
the solvability of which is still an open problem. Therefore, it is not suitable to follow the dynamic programming
principle approach to investigate the underlying problem in our paper.

Another approach relies on the stochastic maximum principle, where the existence of
optimal controls is studied through the stochastic Hamiltonian system. Indeed, the stochastic
Hamiltonian system is a fully-coupled forward-backward stochastic differential equation (FBSDE),
consisting of the state equation, the adjoint equation and the optimality conditions of the optimal control.
Even in the finite-dimensional case, the uniqueness and existence of solutions to nonlinear
fully-coupled forward-backward systems is a very challenging problem. There has been many works on this
topic, see for example, \cite{hu1995solution,ma1994solving,pardoux1999forward,peng1999fully} and the references therein.
However, very limited works have focused on the solvability of infinite-dimensional FBSDEs.
\cite{guatteri2007class} proved that a class of fully-coupled, infinite-dimensional FBSDEs has a local unique solution.
\cite{guatteri2013existence} considered a stochastic optimal control problem for an heat equation
with boundary noise, boundary controls and deterministic coefficients. In \cite{guatteri2013existence},
under suitable assumptions on the coefficients,
the existence condition of optimal controls was presented in strong sense by solving the associated stochastic Hamiltonian system
of infinite dimensions; the bridge method and the auxiliary deterministic Riccati equation were applied to obtain the solution.

In this paper, the stochastic Hamiltonian system is described by the following infinite-dimensional FBSDE:
\begin{eqnarray}\label{eq:1.4}
\left\{
\begin{aligned}
dk(t)=& -\big[A^*(t)k(t)+l_y(t,y(t),z(t),u(t))\big]dt \\
&-\big[B^*(t)k(t)+l_z(t,y(t),z(t),u(t))\big]dW(t),\\
dy(t)=& \ \big [A(t)y(t)+B(t)z(t)+D(t)\gamma( D^*(t)k(t))+G(t)\big ]dt+z(t)dW(t),\\
k(0)=&-h_y(y(0)), \quad y(T)=\xi .
\end{aligned}
\right.
\end{eqnarray}
where $A^*$, $B^*$ and $D^*$ denote the dual operators of $A$, $B$ and $D$,
respectively, and $\gamma$ is a function satisfying suitable conditions,
to be specified below in Assumption (A.5).
Unlike the Hamiltonian system in \cite{guatteri2013existence}, since all the coefficients
in \eqref{eq:1.4} are random and time-varying, the adaptability of the
integrand in the stochastic integral may not be satisfied and the solution of this equation cannot be defined
in the mild sense. Instead, we will study FBSDE \eqref{eq:1.4} in the sense of weak solution (i.e. in the PDE sense).
We first show the existence and uniqueness of a solution to FBSDE \eqref{eq:1.4}
via using continuous dependence theorems for stochastic evolution equations (SEEs) and
backward stochastic evolution equations (BSEEs) in \cite{meng2013stochastic}.
Then from stochastic maximum principle in \cite{meng2013stochastic}, the existence
of an optimal control is immediately obtained. Compared with existing works on infinite-dimensional
FBSDEs (see e.g. \cite{guatteri2013existence}),
the approach developed in our paper is more convenient and much simpler.

The rest of this paper is organized as follows. Section 2 introduces
some basic notation, formulates the control problem in an infinite-dimensional
backward system and recalls stochastic maximum principles established by
\cite{meng2013stochastic}. In Section 3, main results in our paper are provided,
and two infinite-dimensional linear-quadratic control problems are solved in Section 4.
Section 5 concludes the paper with some remarks.

\section{Preliminaries and problem formulation}

In this section, we first introduce the basic notation to be used throughout this paper.
We formulate the control problem under a state equation descirbed by a backward stochastic partial differential
equation (BSPDE) in the abstract evolution form, i.e. a BSEE. At the end of this section,
we give necessary and sufficient maximum principles for our control system.

First of all, we fix a complete probability space $(\Omega, {\mathscr F}, P)$.
Let $W (\cdot) \triangleq \{ W (t) \}_{t \geq 0}$ be a one-dimensional standard Brownian motion defined on
$(\Omega, {\mathscr F}, P)$. We further equip $(\Omega, {\mathscr F}, P)$
with a filtration ${\mathbb F} \triangleq \{ {\mathscr F}_t \}_{t \geq 0}$, which is the natural filtration
generated by $W (\cdot)$ and augmented in the usual way. Denote by  $\mathscr{P}$ the
predictable $\sigma$-field on $[0, T]\times\Omega$, $\mathscr {B} (\Lambda)$
the Borel $\sigma$-algebra of any topological space $\Lambda$,
and $\|\cdot\|_H$ the norm of any Hilbert space $H$. Let $T$ be a finite time horizon, i.e. $0 < T < \infty$.
Throughout this paper, we let $C$ and $K$ be two generic constants, which may be different from line to line.
We introduce the following spaces on $(\Omega, {\mathscr F}, {\mathbb F}, P)$ for Hilbert space-valued
processes or random variables:
\begin{itemize}
\item $M_{\mathscr{F}}^2(0,T;H)$: the set of all
${\mathbb {F}}$-adapted, $H$-valued processes $\varphi=\{\varphi(t,\omega),\
(t,\omega)\in[0,T]\times\Omega\}$ such that
$\|\varphi\|_{M_{\mathscr{F}}^2(0,T;H)} \triangleq \sqrt{\mbox {E} [ \int_0^T\|\varphi(t)\|_H^2dt ] } <\infty$;
\item $S_{\mathscr{F}}^2(0,T;H)$:  the set of all
${\mathbb {F}}$-adapted, $H$-valued, c\`{a}dl\`{a}g processes
$\varphi=\{\varphi(t,\omega),\ (t,\omega)\in[0,T]\times\Omega\}$ such that
$\|\varphi\|_{S_{\mathscr{F}}^2(0,T;H)}\triangleq\sqrt{ \mbox {E} [ \sup_{0 \leq t \leq T}\|\varphi(t)\|_H^2 ] }<+\infty$;
\item $L^2(\Omega,\mathscr{F}_t,P;H)$: the set of all $\mathscr{F}_t$-measurable, $H$-valued random
variables $\xi$ on $(\Omega,{\mathscr{F}},P)$ such that
$\|\xi\|_{L^2(\Omega,\mathscr{F}_t,P;H)} \triangleq \sqrt{\mbox {E} [ \|\xi\|_H^2 ] }<\infty$.
\end{itemize}

In what follows, we introduce the Gelfand triple, in which SEEs
and BSEEs throughout this paper are defined. Let $V$ and $H$ be two separable, real-valued
Hilbert spaces such that $V$ is densely embedded in $H$.
We identify $H$ with its dual space by the Riesz mapping.
Thus, we can take $H$ as a pivot space and  get a
Gelfand triple $(V, H, V^*)$ such that $V \subset H= H^*\subset V^{*},$ where $H^*$ and
$V^{*}$ denote the dual spaces of $H$ and $V$, respectively. Denote
by $(\cdot,\cdot)_H$ the inner product in $H$, and
$\left <\cdot,\cdot\right >$ the duality product between
$V$ and $V^{*}$. Define $\mathscr{L}(V,V^*)$ as the space of bounded
linear transformations from $V$ to $V^*$. With $V$ and $V^*$ being replaced,
other spaces of bounded linear transformations can be defined similarly in the sequel.

We consider the following controlled BSEE in the Gelfand triple $(V, H, V^*)$:
\begin{equation}\label{eq:3.1}
y(t)= \xi-\int_t^T\big [A(s)y(s) +B(s)z(s)+D(s)u(s)+G(s)\big]ds
- \int_t^Tz(s)dW(s) ,
\end{equation}
with the cost functional:
\begin{equation}\label{eq:2.2}
J(u(\cdot))= {\mbox E} \bigg[\int_0^Tl(t,y(t),z(t),u(t))dt+h(y(0))\bigg],
\end{equation}
where $\xi: \Omega \rightarrow H$,
$A:[0,T]\times \Omega\rightarrow \mathscr{L}(V,V^*)$,
$B:[0,T]\times\Omega\rightarrow \mathscr{L}(H, H)$,
$D:[0,T]\times \Omega\rightarrow \mathscr{L}(U, H)$,
$G:[0,T]\times \Omega\rightarrow H$
and $l:[0,T]\times \Omega\times V\times H\times U \rightarrow \mathbb{R}$,
$h :\Omega\times V \rightarrow \mathbb{R}$ are given random mappings.
Suppose that the control set $U$ is a separable Hilbert space and is a convex set.
An $\mathbb{F}$-adapted, $U$-valued process $u(\cdot)$ such that $\mbox {E}
[ \int_0^T \| u(t) \|_{U}^{2} d t ]<\infty$ is called an \emph{admissible control}. Denote by
$\cal A$ the set of all admissible controls.

In Section 4, $A$ and $B$ will be specified by the second-order and the first-order differential operators
and, meanwhile, the control problem will turns out to be a Dirichlet problem for BSPDEs. This can facilitate the understanding
of the abstract evolution form \eqref{eq:3.1}-\eqref{eq:2.2}. One may also refer to \cite{meng2013stochastic}
for a Cauchy problem for BSDPEs.

The Hamiltonian function
\begin{eqnarray*}
{\cal H}: [0, T] \times \Omega \times V \times H \times U \times V \rightarrow \mathbb {R}
\end{eqnarray*}
of the control system
\eqref{eq:3.1}-\eqref{eq:2.2} is defined by
\begin{eqnarray}\label{eq:4.2}
{\cal H} (t,y,z,u,k)= \big ( B(t)z+D(t)u, k \big )_H+l(t,y,z,u).
\end{eqnarray}

Let us make the following assumptions on the coefficients of the control system
\eqref{eq:3.1}-\eqref{eq:2.2}:
\begin{enumerate}

\item[\bf (A.1)] The terminal value $\xi\in L^2(\Omega,\mathscr{F}_T,P;H)$,  $ B$  and  $D $ are  uniformly bounded ${\mathscr{F}}_t$-predictable processes, and
    $G$  is ${\mathscr{F}}_t$-predictable processes
    with $G \in M_{\mathscr{F}}^2(0,T;H).$

\item[\bf (A.2)] The operator $A$ satisfies the following coercivity and boundedness conditions:
(i) there exist constants $\alpha >0$ and $\lambda$ such that
\begin{eqnarray*}
\left < A(t)y,y \right > + \lambda\|y\|^2_H\geq\alpha\| y\|^2_V,
~~~~~\forall t\in [0, T],~~\forall y\in V,
\end{eqnarray*}
and (ii) there exists a constant $C >0$ such that
\begin{eqnarray*}
\sup_{(t,\omega)\in [0, T]\times \Omega} \|A(t,\omega)\|_{\mathscr{L}(V, V^{*})} \leq C.
\end{eqnarray*}

\item[\bf (A.3)] The map $l$ is ${\mathscr{P}}\otimes {\mathscr B}(V)\otimes
{\mathscr B}(H)\otimes {\mathscr B}(U)/ {\mathscr B}(\mathbb{R})$-measurable and for
almost all $(t, \omega)\in [0, T]\times \Omega$, $l(t,\omega,y,z,u)$
is convex and  G\^{a}teaux  differentiable in $( y,z,u)$ with continuous
G\^{a}teaux derivatives $ l_y, l_z, l_u$.
The map $h$ is ${\mathscr{F}}_0 \otimes {\mathscr B}(V)/ {\mathscr B}(\mathbb{R})$-measurable
and for almost all $\omega \in\Omega$, $h(\omega,y)$ is
convex and G\^{a}teaux differentiable in $y$ with continuous
G\^{a}teaux derivative $h_y$. Moreover, for almost all
$(t,\omega)\in [0, T]\times \Omega$, there exists a constant $C > 0$
such that, for all $(y,z,u)\in V\times H\times U$,
$$|l(t,y,z,u)|\leq C(1+\|y\|^2_V+\|z\|^2_H+\|u\|_U^2),$$
$$\|l_y(t,y,z,u)\|_V+\|l_z(t,y,z,u)\|_H+\|l_u(t,y,z,u)\|_U\leq
C(1+\|y\|_V+\|z\|_H+\|u\|_U),$$
and
$$|h(y)|\leq C(1+\|y\|^2_V),  $$
$$\|h_y(y)\|_V\leq C(1+\|y\|_V).$$

\item[\bf (A.4)] For almost all $(t,\omega)\in [0, T]\times \Omega$,
there exists a constant $C > 0$ such that, for all $y_1,y_2\in V$, $z_1,z_2\in H$, $u\in U$,
\begin{eqnarray*}
&& ( l_y(t,y_1,z_1,u)-l_y(t,y_2,z_2, u), y_1-y_2 )_H \\
&& + ( l_z(t,y_1,z_1,u)-l_z(t,y_2,z_2, u), z_1-z_2 )_H
\geq C ( ||y_1-y_2||^2_V+ ||z_1-z_2||^2_H ) ,
\end{eqnarray*}
and
\begin{eqnarray*}
( h_y(y_1)-h_y(y_2), y_1-y_2 )_H \geq C ||y_1-y_2||^2_V.
\end{eqnarray*}

\item[\bf (A.5)]
For all $(t, y, z, u, k) \in [0, T] \times V \times H \times U \times V$, there exists a function
$\gamma: U \rightarrow U$ such that
\begin{equation}\label{eq:4.2}
\begin{array}{ll} \displaystyle
{\cal H} (t, y,z,\gamma (D^*(t)k), k)= \min_{u \in U} {\cal H} (t,y,z,u,k).
\end{array}
\end{equation}
For all $k_1,k_2 \in V$, there exists a constant $C > 0$ such that
\begin{eqnarray}
& \big ( D(t)\gamma(D^*(t)k_1)-D(t)\gamma (D^*(t)k_2), k_1-k_2 \big )_H \leq  0 , \\
& ||D(t)\gamma(D^*(t)k_1)-D(t)\gamma(D^*(t)k_2)||_H \leq C ||k_1-k_2||_V,
\end{eqnarray}
where $D^*: [0,T]\times \Omega\rightarrow \mathscr{L}(H, U)$ is the dual operator of $D$.
\end{enumerate}

In what follows, let {\bf Assumption A} stand for Assumptions {\bf (A.1)}-{\bf (A.5)}.
Under {\bf Assumption A}, it follows from Theorem 4.1 in \cite{hu1991adapted}
or Theorem 2.2 in \cite{peng1992stochastic} that the system \eqref{eq:3.1} admits
a unique solution $(y(\cdot), z(\cdot)) \in M_{\mathscr{ F}}^2(0,T;V) \times M_{\mathscr{F}}^2(0,T;H)$,
for each  $u(\cdot)\in {\cal A}$. Whenever we need to stress the dependence on the control
$u (\cdot)$, we denote by $(y^u(\cdot), z^u(\cdot)) : = (y(\cdot), z(\cdot))$ in the sequel.
Then, we call $(y^u (\cdot), z^u (\cdot) )$ the state process corresponding to the control process
$u(\cdot)$ and $(u(\cdot); y(\cdot), z(\cdot))$ the admissible
pair. Furthermore, from {\bf Assumption A},  we can easily check that
\begin{eqnarray}
|J(u(\cdot))|<\infty.
\end{eqnarray}

We now state the optimal control problem to be considered:
\begin{problem}\label{pro:2.1}
Find an admissible control $ {u}(\cdot) \in {\cal A}$ such that
\begin{eqnarray}\label{eq:b77}
J( {\bar u}(\cdot))=\displaystyle\inf_{u(\cdot)\in {\cal
A}}J(u(\cdot)).
\end{eqnarray}
\end{problem}
Any ${\bar u}(\cdot)\in {\cal A}$ satisfying Eq. (\ref{eq:b77}) is called an
optimal control of Problem \ref{pro:2.1} and the
corresponding state process $( {\bar y}(\cdot),  {\bar z}(\cdot))$ is called
an optimal state process. Correspondingly, $( {\bar u}(\cdot);
{\bar y}(\cdot), {\bar z}(\cdot))$ is called an optimal pair of
Problem \ref{pro:2.1}.

For any  given admissible pair $({u}(\cdot);{y}(\cdot), {z}(\cdot))$,
we consider the following adjoint equation:
\begin{eqnarray}\label{eq:f333}
\left\{
\begin{aligned}
dk(t)=& -\big [A^*(t)k(t)+l_y(t,  y(t),  z(t),u(t))\big ]dt \\
& -\big [B^*(t)  k(t)+l_z(t,  y(t),  z(t),u(t))\big]dW(t), \\
k(0)=& -h_y(  y(0)) ,
\end{aligned}
\right.
\end{eqnarray}
where $A^*: [0,T]\times \Omega\rightarrow \mathscr{L}(V^*,V)$
and $B^*: [0,T]\times \Omega\rightarrow \mathscr{L}(H,H)$ denote the dual operators of $A$ and $B$,
respectively. Indeed, the adjoint equation \eqref{eq:f333} is a linear SEE.
Under {\bf Assumptions A}, by Theorem I in \cite{bensoussan1983stochastic}, it can be shown
that the above adjoint equation admits a unique solution
$k(\cdot)\in M_{\mathscr{F}}^2(0,T; V)$.

\begin{theorem}\label{thm:4.2}
Given that {\bf Assumption A} is satisfied.
Let $( {u}(\cdot);  {y}(\cdot), {z}(\cdot))$  be
an optimal pair of  Problem \ref{pro:2.1} and
$k(\cdot)$ be the solution of the adjoint equation \eqref{eq:f333}
associated with $( {u}(\cdot);  {y}(\cdot), {z}(\cdot))$. Then we have
\begin{eqnarray}\label{eq:4.15}
\big ( {\cal H}_u(t,  y(t),   z(t),   u(t), k(t)), u- {u}(t) \big )_U \geq 0,
\end{eqnarray}
for all $u\in U$, a.e. $t \in [0, T]$, $P$-a.s..
\end{theorem}

\begin{theorem}\label{thm:4.1}
Given that {\bf Assumption A} is satisfied. Let $( {u}(\cdot);
{y}(\cdot),  {z}(\cdot))$ be an admissible pair and
$k(\cdot)$ be the unique solution of the corresponding adjoint
equation \eqref{eq:f333}. If for almost all $(t,\omega)\in
[0,T]\times \Omega$, ${\cal H} (t,y,z,u,{k}(t))$ is convex in $(y,z,u)$,
$h(y)$ is convex in $y$ and the following optimality condition holds
\begin{eqnarray}\label{eq:5.119}
{\cal H} (t, {y}(t),  z(t), {u}(t),  {k}(t))
= \min_{u\in U} {\cal H} (t, {y}(t),  z(t), u,{k}(t)) ,
\end{eqnarray}
then $u (\cdot)$ is the optimal control of Problem \ref{pro:2.1}
and $( {u}(\cdot);   {y}(\cdot),  {z}(\cdot))$ is the optimal pair.
\end{theorem}

Theorems \ref{thm:4.2} and \ref{thm:4.1} are called necessary maximum principle and sufficient
maximum principle (or verification theorem) for optimality of the control system
(\ref{eq:3.1})-(\ref{eq:2.2}), which were obtained by \cite{meng2013stochastic}.

\section{Main Results}

In this section, we first prove that the stochastic Hamiltonian
system admits a unique solution. Using the maximum principle, then we show
that Problem \ref{pro:2.1} has a unique optimal control and thus a unique optimal control
pair. This is the main result of our paper.

First of all, we restate the stochastic Hamiltonian system associated with our optimal control problem:
\begin{eqnarray}\label{eq:4.4}
\left\{
\begin{aligned}
dk(t)=& -\big[A^*(t)k(t)+l_y(t,y(t),z(t),u(t))\big]dt \\
& -\big[B^*(t)k(t)+l_z(t,y(t),z(t),u(t))\big]dW(t),\\
dy(t)=& \ \big [A(t)y(t)+B(t)z(t)+D(t)\gamma( D^*(t)k(t))+G(t)\big ]dt+z(t)dW(t),\\
k(0)=&-h_y(y(0)), \quad y(T)=\xi,~~~~~~~t\in [0, T].
\end{aligned}
\right.
\end{eqnarray}
Indeed, the stochastic Hamiltonian system is a forward-backward stochastic partial differential
equation (FBSPDE) or a forward-backward stochastic evolution equation (FBSEE), which is fully
coupled.

In what follows, we denote by
\begin{eqnarray*}
\mathbb{M}^2[0,T] \triangleq M_{\mathscr{F}}^2(0,T;V)\times
M_{\mathscr{F}}^2(0,T; V) \times M_{\mathscr{F}}^2(0,T; H) .
\end{eqnarray*}
Clearly, $\mathbb{M}^2[0,T]$ is a Banach space equipped with the following norm:
\begin{eqnarray*}
\| \big ( k (\cdot), y (\cdot), z (\cdot) \big ) \|_{\mathbb{M}^2[0,T]} \triangleq
\bigg \{ \|k (\cdot)\|^2_{M_{\mathscr{F}}^2(0,T;V)} + \|y (\cdot)\|^2_{M_{\mathscr{F}}^2(0,T;V)}
+ \|z (\cdot)\|^2_{M_{\mathscr{F}}^2(0,T;H)} \bigg \}^{\frac{1}{2}} .
\end{eqnarray*}

The following theorem confirms the existence and uniqueness of a solution to the forward-backward system \eqref{eq:4.4}.
This result will play a vital role in proving the existence of the optimal control.

\begin{theorem}\label{thm:4.3}
Given that {\bf Assumption  A} is satisfied. There exists a unique solution
$(k(\cdot), y(\cdot),z(\cdot))\in \mathbb {M}^2[0, T]$ of the Hamiltonian system \eqref{eq:4.4}.
\end{theorem}

Before proving Theorem \ref{thm:4.3}, we state and prove the main result of the paper,
i.e. the existence of an optimal control for the BSPDE control system \eqref{eq:3.1}-\eqref{eq:2.2}.
Once we have proved Theorem \ref{thm:4.3}, the main result is an immediate consequence of
Theorem \ref{thm:4.1}. The proof of Theorem \ref{thm:4.3} will be postponed
after we present the following main result.

\begin{theorem} \label{thm:4.4}
Given that {\bf Assumption A} is satisfied. There exists a unique optimal control
$\gamma(D^*(\cdot) k(\cdot))$ and thus a unique optimal control pair $(\gamma(
D^*(\cdot)  k(\cdot));   y(\cdot), z(\cdot))$ of Problem \ref{pro:2.1}.
\end{theorem}
\begin{proof}
From Theorem \ref{thm:4.3}, the Hamiltonian system \eqref{eq:4.4} admits a unique solution.
Let $( k(\cdot),  y(\cdot),  z(\cdot))$ be this unique solution. By the definition of the map $\gamma$,
we know that $(\gamma(D^*(\cdot)  k(\cdot));  y(\cdot), z(\cdot))$ is an admissible
pair and $k(\cdot)$ is the corresponding adjoint process. By {\bf Assumption A}, we have
\begin{eqnarray}\label{eq:5.119}
{\cal H} (t, {y}(t),  z(t), \gamma( D^*(t)  k(t)),  {k}(t))
= \min_{u\in U} {\cal H} (t, {y}(t),  z(t), u,{k}(t)).
\end{eqnarray}
Using Theorem \ref{thm:4.1}, we conclude that $\gamma(D^*(\cdot)  k(\cdot))$ is the unique optimal control
and $(\gamma(D^*(\cdot)  k(\cdot));  y(\cdot), z(\cdot))$ is the unique optimal control pair of Problem \ref{pro:2.1}.
\end{proof}

To prove Theorem \ref{thm:4.3}, we consider the following auxiliary FBSEE:
\begin{eqnarray}\label{eq:11}
\left\{
\begin{aligned}
dk(t) =& - \big [ A^*(t)k(t)+\rho l_y(t,y(t),z(t),u(t))+(1-\rho) Cy(t) +b_0(t)\big ]dt \\
& - \big [ B^*(t)k(t)+\rho l_z(t,y(t),z(t),u(t)) +(1-\rho) Cz(t)+g_0(t) \big ]dW (t), \\
dy (t) =& \ \big [ A(t)y(t)+ B(t)z(t)+\rho D(t)\gamma( D^*(t)k(t)) +G(t)+f_0(t) \big ]dt +z(t)dW (t),\\
k(0) =& -h_y(y(0)) , \quad y(T)=\xi,
\end{aligned}
\right.
\end{eqnarray}
where $b_0 (\cdot)\in M^2_{\cal F}(0,T;V), g_0 (\cdot), f_0 (\cdot) \in M^2_{\cal F}(0,T;H)$ and $\rho \in [0, 1]$.
The next lemma discusses the solvability of FBSEE \eqref{eq:11}.

\begin{lemma} \label{lem:4.3}
Given that {\bf Assumption A} is satisfied. Suppose that
for any $b_0 (\cdot)\in M^2_{\cal F}(0,T;V)$ and $g_0 (\cdot), f_0 (\cdot) \in M^2_{\cal F}(0,T;H)$,
FBSEE \eqref{eq:11} associated with some $\rho=\rho_0$
admits a unique solution $(k(\cdot), y(\cdot),z(\cdot))\in \mathbb {M}^2[0, T]$.
Then there exists $\delta_0\in (0,1]$ such that for any $\rho\in [\rho_0, \rho_0+\delta_0]$,
FBSEE \eqref{eq:11} admits a solution $(k(\cdot), y(\cdot),z(\cdot))\in \mathbb {M}^2[0,T]$.
\end{lemma}

\begin{proof}
For any $\rho \in [0,1]$ other than $\rho_0$, we can rewrite FBSEE (\ref{eq:11}) as
\begin{eqnarray}\label{eq:13}
\left\{
\begin{aligned}
dk(t) =& - \big [A^*(t)k(t)+\rho_0 l_y(t,y(t),z(t),u(t))+ (1-\rho_0)Cy(t) \\
& +(\rho-\rho_0) l_y(t,y(t),z(t),u(t)) -(\rho-\rho_0)Cy(t) +b_0(t) \big ]dt \\
& - \big [B^*(t)k(t)+\rho_0 l_z(t,y(t),z(t),u(t)) +(1-\rho_0) Cz(t) \\
&+(\rho-\rho_0) l_z(t,y(t),z(t),u(t)) -(\rho-\rho_0) Cz(t) + g_0(t) \big ]dW (t), \\
dy (t) =& \ \big [A(t)y(t)+
B(t)z(t)+\rho_0 D(t)\gamma( D^*(t)k(t)) \\
& +(\rho-\rho_0) D(t)\gamma( D^*(t)k(t))+G(t) +f_0(t)\big ]dt+z(t)dW (t),\\
k(0)=& \ -h_y(y(0)), \quad y(T)=\xi .
\end{aligned}
\right.
\end{eqnarray}
Thus for any $\Lambda^\prime (\cdot)=(  k^\prime (\cdot),  y^\prime (\cdot), z^\prime (\cdot))
\in \mathbb {M}^2[0, T]$, the following FBSEE
\begin{eqnarray}\label{eq:12}
\left\{
\begin{aligned}
dk(t) =& - \big [A^*(t)k(t)+\rho_0 l_y(t,y(t),z(t),u(t))+(1-\rho_0)Cy(t) \\
& +(\rho-\rho_0) l_y(t,y^\prime (t),z^\prime (t),u(t)) -(\rho-\rho_0)Cy^\prime (t) +b_0(t) \big ]dt \\
& - \big [B^*(t)k(t)+\rho_0 l_z(t,y(t),z(t),u(t)) +(1-\rho_0) Cz(t) \\
&+(\rho-\rho_0) l_z(t,y^\prime (t),z^\prime (t),u(t))
-(\rho-\rho_0) Cz^\prime (t) +g_0(t) \big ]dW (t), \\
dy (t) =& \ \big [A(t)y(t)+
B(t)z(t)+\rho_0 D(t)\gamma( D^*(t)k(t)) \\
& +(\rho-\rho_0) D(t)\gamma( D^*(t)k^\prime (t)) +G(t)+f_0(t)\big ]dt+z(t)dW (t),\\
k(0)=& \ -h_y(y(0)), \quad y(T)=\xi ,
\end{aligned}
\right.
\end{eqnarray}
has a unique solution $\Lambda (\cdot)=(k(\cdot),y(\cdot), z(\cdot))\in \mathbb {M}^2[0, T]$.
Hence, by FBSEE (\ref{eq:12}), we can define a mapping ${\cal I}: \mathbb {M}^2[0, T] \rightarrow
\mathbb {M}^2[0, T]$ such that ${\cal I} ( \Lambda^\prime (\cdot)) = \Lambda (\cdot)$.

Next we claim that ${\cal I}$ is a contraction mapping. In fact,
for any $\Lambda^\prime_i (\cdot) =(  k^\prime_i(\cdot),  y^\prime_i(\cdot), z^\prime_i(\cdot))
\in  \mathbb {M}^2[0, T]$, $i=1, 2$, we can define $\Lambda_i (\cdot) =
(k_i(\cdot),y_i(\cdot), z_i(\cdot)))= {\cal I} ( \Lambda^\prime_i (\cdot))$.
On the one hand, from {\bf Assumption A} and the continuous dependence theorem
for SEEs (see Lemma 2.3 in \cite{meng2013stochastic}), we have
\begin{eqnarray}\label{eq:3.11}
&& {\mbox E} \bigg [ \sup_{0\leq t\leq T}\|{k_1} (t)-k_2(t)\|^2_H \bigg ]
+ {\mbox E} \bigg [ \int_{0}^T\| {k_1} (t)-k_2(t)\|_V^2dt \bigg ] \nonumber \\
&& \leq K \bigg \{ {\mbox E} \big [ ||y_1(0)- y_2(0)||^2_H \big ]
+ {\mbox E} \bigg [ \int_{0}^T\| {y_1} (t)-y_2(t)\|_V^2dt \bigg ] \nonumber \\
&& \quad + {\mbox E} \bigg [ \int_{0}^T\| {z_1} (t)-z_2(t)\|_H^2dt \bigg ]
+ |\rho-\rho_0| \cdot || \Lambda^\prime_1(\cdot)- \Lambda^\prime_2(\cdot) ||^2_{\mathbb {M}^2[0, T]} \bigg \} \ .
\end{eqnarray}
On the other hand, from {\bf Assumption A} and the continuous dependence theorem
for BSEEs (see Lemma 2.5 in \cite{meng2013stochastic}), we have
\begin{eqnarray}\label{eq:4.10}
&& {\mbox E} \bigg [ \sup_{0\leqslant t\leqslant T}\|{y_1}(t)-y_2(t)\|^2_H \bigg ]
+ {\mbox E} \bigg [ \int_{0}^T\| {y_1} (t)-y_2(t)\|_V^2dt \bigg ]
+ {\mbox E} \bigg [ \int_{0}^T\| {z_1} (t)-z_2(t)\|_H^2dt \bigg ] \nonumber \\
&& \leq K \bigg \{ {\mbox E} \bigg [ \int_{0}^T\| {k_1} (t)-k_2(t)\|_V^2dt \bigg ]
+ |\rho-\rho_0| \cdot || \Lambda^\prime_1(\cdot)- \Lambda^\prime_2(\cdot)||^2_{\mathbb {M}^2[0, T]} \bigg \} .
\end{eqnarray}
Furthermore, applying It\^{o}'s formula to $\big ( k_1 (t)-k_2(t),y_1(t)-y_2(t) \big )_H$ (please refer to
\cite{krylov1981stochastic} for a version of It\^{o}'s formula in Hilbert spaces) and
noting the duality relations between $A$, $B$ and $A^*$, $B^*$, we deduce
\begin{eqnarray}\label{eq36}
&& {\mbox E} \big [ \big ( h_y(y_1 (0))-h_y({y_2} (0), y_1(0)-y_2(0) \big )_H \big ] \nonumber \\
&& = -\rho_0 {\mbox E} \bigg [ \int_0^T \big ( l_y(t,y_1(t),z_1(t),u(t))
-l_y(t,y_2(t),z_2(t),u(t)),y_1(t)-y_2(t) \big )_H dt \bigg ] \nonumber \\
&& \quad -\rho_0 {\mbox E} \bigg [ \int_0^T \big ( l_z(t,y_1(t), z_1(t),u(t))-l_z(t,y_2(t),
z_2(t),u(t)),z_1(t)-z_2(t) \big )_H dt \bigg ] \nonumber \\
&& \quad -(\rho-\rho_0) {\mbox E} \bigg [ \int_0^T \big ( l_y(t,
y_1^\prime(t), z_1^\prime(t),u(t))-l_y(t, y_2^\prime(t),
z_2^\prime (t),u(t)), y_1(t)-y_2(t) \big )_H dt \bigg ] \nonumber \\
&& \quad -(\rho-\rho_0) {\mbox E} \bigg [ \int_0^T \big ( l_z(t, y_1^\prime
(t), z_1^\prime(t),u(t))-l_z(t,  y_2^\prime(t),  z_2^\prime(t),u(t)), z_1(t)-z_2(t) \big )_H dt \bigg ] \nonumber \\
&& \quad -(1-\rho_0)C {\mbox E} \bigg [ \int_0^T ||y_1(t)-y_2(t)||_V^2dt \bigg ]
+(\rho-\rho_0)C {\mbox E} \bigg [ \int_0^T\big (y_1^\prime(t)-y_2^\prime(t), y_1(t)-  y_2(t)\big )_H dt \bigg ] \nonumber \\
&& \quad -(1-\rho_0)C {\mbox E} \bigg [ \int_0^T|| z_1(t)-z_2(t)||_H ^2dt \bigg ] + (\rho-\rho_0)C
{\mbox E} \bigg [ \int_0^T \big (z_1^\prime(t)-z_2^\prime(t),z_1(t)-  z_2(t)\big )_H dt \bigg ] \nonumber \\
&& \quad +\rho_0 {\mbox E} \bigg [ \int_0^T \big ( D(t)\gamma( D^*(t)k_1(t))-D(t)\gamma(
D^*(t)k_2(t)), k_1(t)-k_2(t) \big )_H dt \bigg ] \nonumber \\
&& \quad +(\rho-\rho_0) {\mbox E} \bigg [ \int_0^T \big ( D(t)\gamma(
D^*(t)  k_1^\prime(t))-D(t)\gamma( D^*(t)  k_2^\prime(t)), k_1(t)-k_2(t)\big )_H dt \bigg ] .
\end{eqnarray}
By the monotonicity conditions (see {\bf Assumption A}), we get
\begin{eqnarray}
&& C {\mbox E} \big [ \|{y_1} (0))-y_2(0)\|^2_V \big ]
+ C {\mbox E} \bigg [ \int_0^T \|{y_1} (t))-y_2(t)\|^2_V d t \bigg ]
+ C {\mbox E} \bigg [ \int_0^T\|{z_1} (t))-z_2(t)\|^2_H d t \bigg ] \nonumber \\
&& \leq -(\rho-\rho_0) {\mbox E} \bigg [ \int_0^T \big ( l_y(t,  y_1^\prime(t),
z_1^\prime(t),u(t))-l_y(t, y_2^\prime(t), z_2^\prime(t),u(t)), y_1(t)-y_2(t) \big )_H dt \bigg ] \nonumber \\
&& \quad -(\rho-\rho_0) {\mbox E} \bigg [ \int_0^T \big (
l_z(t, y_1^\prime(t), z_1^\prime(t),u(t))-l_z(t,  y_2^\prime(t), z_2^\prime(t),u(t)),
z_1(t)-z_2(t) \big )_H dt \bigg ] \nonumber \\
&& \quad +(\rho-\rho_0) C
{\mbox E} \bigg [ \int_0^T \big ( y_1^\prime(t)-y_2^\prime(t), y_1(t)-  y_2(t) \big )_H dt \bigg ] \nonumber \\
&& \quad +(\rho-\rho_0)C
{\mbox E} \bigg [ \int_0^T \big ( z_1^\prime(t)-z_2^\prime(t), z_1(t)-  z_2(t)\big )_H dt \bigg ]
\nonumber \\
&& \quad +(\rho-\rho_0)
{\mbox E} \bigg [ \int_0^T \big ( D(t)\gamma( D^*(t) k_1^\prime(t))-D(t)\gamma(
D^*(t)  k_2^\prime(t)), k_1(t)-k_2(t) \big )_H dt \bigg ] .
\end{eqnarray}
Using the elementary equality $2ab \leq \frac {1}{\varepsilon} a^2+
\varepsilon b^2$, where $\varepsilon$ is a constant satisfying
$\varepsilon \in (0, C)$, we have
\begin{eqnarray}\label{eq:4.13}
&& C {\mbox E} \big [ \|{y_1} (0))-y_2(0)\|^2_V \big ]
+( C-\varepsilon) {\mbox E} \bigg [ \int_0^T\|{y_1}(t))-y_2(t)\|^2_V d t \bigg ]
+( C-\varepsilon) {\mbox E} \bigg [ \int_0^T\|{z_1}(t))-z_2(t)\|^2_H d t \bigg ] \nonumber \\
&&\leq \varepsilon {\mbox E} \bigg [ \int_{0}^T\| {k_1} (t)-k_2(t)\|_H^2dt \bigg ] +
K |\rho-\rho_0| \cdot || \Lambda^\prime_1(\cdot)- \Lambda^\prime_2(\cdot) ||^2_{\mathbb {M}^2[0, T]}.
\end{eqnarray}
Hence, taking a sufficiently small $\varepsilon$
and putting \eqref{eq:4.13} into \eqref{eq:3.11} give
\begin{eqnarray}\label{eq:4.14}
{\mbox E} \bigg [ \sup_{0\leq t\leq T}\|{k_1}(t)-k_2(t)\|^2_H \bigg ]
+ {\mbox E} \bigg [ \int_{0}^T\| {k_1} (t)-k_2(t)\|_V^2dt \bigg ]
\leq K |\rho-\rho_0| \cdot || \Lambda^\prime_1(\cdot)- \Lambda^\prime_2(\cdot) ||^2_{\mathbb {M}^2[0, T]}.
\end{eqnarray}
Here the positive constant $K$ depends only on $C$, $\varepsilon$, $T$, $\alpha$ and $\lambda$.
Putting \eqref{eq:4.14} into \eqref{eq:4.10}, we obtain
\begin{eqnarray}\label{eq:4.15}
&& {\mbox E} \bigg [ \sup_{0\leq t\leq T}\|{y_1}(t)-y_2(t)\|^2_H \bigg ]
+{\mbox E} \bigg [ \int_{0}^T\| {y_1}(t)-y_2(t)\|_V^2dt \bigg ]
+{\mbox E}\bigg [ \int_{0}^T\| {z_1}(t)-z_2(t)\|_H^2dt \bigg ] \nonumber \\
&& \leq K |\rho-\rho_0| \cdot || \Lambda^\prime_1(\cdot)- \Lambda^\prime_2(\cdot)||^2_{\mathbb {M} ^2[0, T]} .
\end{eqnarray}
Combining \eqref{eq:4.14} and \eqref{eq:4.15} yields
\begin{eqnarray*}
\|{\cal I} ( \Lambda_1(\cdot))- {\cal I} ( \Lambda_2(\cdot))\|_{\mathbb {M}^2[0, T]}\leq K|\rho-\rho_0|
\cdot \| \Lambda^\prime_1(\cdot)- \Lambda^\prime_2(\cdot) \|^2_{{\mathbb M}^2[0, T]} .
\end{eqnarray*}
Recall that $K$ is a positive constant independent of $\rho$ and set $\delta_0 = (2K)^{-1} \wedge 1$.
Then the mapping ${\cal I}$ is contractive in ${\mathbb M}^2[0, T]$
as long as $|\rho-\rho_0|\leq \delta_0$. When $|\rho-\rho_0|\leq \delta_0$,
the contraction mapping theorem implies that FBSEE \eqref{eq:11}
admits a unique solution $(k (\cdot), y (\cdot), z (\cdot))$ in ${\mathbb M}^2[0, T]$.
This completes the proof.
\end{proof}

\begin{proof}[Proof of Theorem \ref{thm:4.3}]

\emph{Existence.}
The proof of the existence can be obtained directly by Lemma \ref{lem:4.3}.
Indeed when $\rho=0$, Eq. \eqref{eq:11} is a decoupled
FBSEE, the uniqueness and existence of which is guaranteed by Theorem I in
\cite{bensoussan1983stochastic} and Theorem 4.1 in \cite{hu1991adapted}
or Theorem 2.2 in \cite{peng1992stochastic}. Starting from $\rho=0$, one can reach $\rho=1$
in finite steps by Lemma \ref{lem:4.3}. Therefore, setting $\rho = 1$ and $b_0 (\cdot) = g_0 (\cdot) = f_0 (\cdot)
= 0$ in the auxiliary FBSEE \eqref{eq:11} proves the existence of a solution
to FBSEE \eqref{eq:4.4}.

\emph{Uniqueness.} Let $(k_i (\cdot),y_i (\cdot), z_i (\cdot))$, for $i=1,2$,
be two solutions of \eqref{eq:4.4}. Using It\^{o}'s formula to
$\big (  k_1 (t)-k_2 (t), y_1(t)-y_2(t) \big )_H$ gives
\begin{eqnarray}\label{eq36}
&& {\mbox E} \big [ \big (  h_y(y_1 (0))-h_y({y_2} (0), y_1(0)-y_2(0) \big )_H \big ] \nonumber \\
&& = - {\mbox E} \bigg [ \int_0^T \big ( l_y(t,y_1(t),z_1(t),u(t))-l_y(t,y_2(t),z_2(t),u(t)),
y_1(t)-y_2(t) \big )_H dt \bigg ] \nonumber \\
&& \quad - {\mbox E} \bigg [ \int_0^T \big ( l_z(t,y_1(t), z_1(t),u(t))-l_z(t,y_2(t),
z_2(t),u(t)), z_1(t)-z_2(t) \big )_H dt \bigg ] \nonumber \\
&& \quad + {\mbox E} \bigg [ \int_0^T \big ( D(t)\gamma( D^*(t)k_1(t))-D(t) \gamma(D^*(t)k_2(t)),
k_1(t)-k_2(t) \big )_H dt \bigg ] .
\end{eqnarray}
Using the monotonicity conditions in {\bf Assumption A} and Eq. (\ref{eq36}) lead to
\begin{eqnarray*}
&& C {\mbox E} \big [ || y_1 (0) - y_2 (0) ||_V^2 \big ]
+ C {\mbox E} \bigg [ \int^T_0 ||y_1 (t)-y_2 (t)||_V^2 d t \bigg ]
+ C {\mbox E} \bigg [ \int^T_0 ||z_1(t)-z_2(t)||_H^2 d t \bigg ] \leq 0.
\end{eqnarray*}
Thus, $y_1(t)\equiv y_2(t), z_1(t)\equiv z_2(t)$.
Finally, from the uniqueness of SEE (see Theorem I in \cite{bensoussan1983stochastic}),
it follows from the forward part of Eq. (\ref{eq:4.4}) that $k_1(t)\equiv k_2(t)$. The proof is complete.
\end{proof}

\section{Examples}

In this section, we illustrate our results with two example of linear-quadratic stochastic optimal control problems
in infinite dimensions.
We reiterate that the state of the control system is given by the linear BSEE  \eqref{eq:3.1}, that is,
\begin{eqnarray}\label{eq:4.22}
\left\{
\begin{aligned}
dy(t)=& \ [A(t)y(t)+B(t)z(t)+D(t)u(t)+G(t)]dt+z(t)dW(t),\\
y(T)=& \ \xi ,
\end{aligned}
\right.
\end{eqnarray}
Moreover, we adopt the following specification:
\begin{eqnarray}\label{eq:6.1}
l(t,y,z,u)=( M(t) y,  y)_H +(Q(t) z,  z)_H +( N(t) u,  u)_U, \quad h(y)=( hy,  y)_H .
\end{eqnarray}
Then the cost functional is given by
\begin{eqnarray}\label{eq:6.1}
J(u(\cdot)):= {\mbox E} \bigg [ \int_0^T(M(s)y(s),y(s))_H ds
+ \int_0^T ( Q(s)z(s),z(s))_H ds + \int_0^T (N(s)u(s),u(s))_U ds + (hy, y)_H \bigg ] .
\end{eqnarray}
Here $M$, $Q$, $N$ and $h$ are given random mappings such that
$M:[0,T]\times \Omega \rightarrow \mathscr{L}(V, H)$,
$Q:[0,T]\times \Omega \rightarrow \mathscr{L}(H, H)$,
$N:[0,T]\times \Omega \rightarrow \mathscr{L}(U, U)$
and $h:\Omega\rightarrow \mathscr{L}(V, H)$.

\begin{problem}\label{pro:5.1}
Find an admissible control $\bar{u}(\cdot)$ such that
\begin{equation*}\label{eq:b7}
J(\bar{u}(\cdot))=\displaystyle\inf_{u(\cdot)\in {
M^2_{\mathscr F}(0,T;U)}}J(u(\cdot)).
\end{equation*}
subject to \eqref{eq:4.22} and \eqref{eq:6.1}.
\end{problem}

To place Problem \ref{pro:5.1} in the general framework considered
in Sections 2-3, we impose the following assumptions on the coefficients:

\begin{assumption}\label{ass:5.1}
The coefficients
$\xi, A, B,D $ and  $G $ satisfy Assumptions $\bf {(A.1)}$  and $\bf {(A.2)}$
\end{assumption}

\begin{assumption}\label{ass:5.2}
The stochastic processes $N$, $M$, $Q$ and the random
variable $h$ are a.e. and a.s. uniformly positive operators,
i.e. for any $u\in U,  y\in H, z\in H$, there exists some positive constant
$\delta$ such that $( N(t)u, u)_U\geq \delta ( u, u)_U$,
$(M(t)y, y)\geq \delta ( y, y)_H$, $(Q(t)z, z)\geq \delta ( z, z)_H$,
and $(hy, y)\geq \delta ( y, y)_H$.
\end{assumption}

The Hamiltonian $\cal H$ of Problem \ref{pro:5.1} is now given by
\begin{eqnarray}\label{eq:6.7}
{\cal H} (t,y,z,u, k)= (B(t)z+D(t)u, k)_H  +( M(t)y, y)_H +( Q(t)z,z)_H +( N(t)u, u)_U.
\end{eqnarray}
Since the Hamiltonian is quadratic with respect
 to $u \in U$ and $N$ is strictly positive,
 the minimum value of the Hamiltonian $\cal H$
 with respect to $u \in U $ can be reached at $-\frac{1}{2}N^{-1}(t) D^*(t)k$.
Therefore, we can define a map $\gamma: U \rightarrow U$ as
\begin{eqnarray*}
\gamma(u)=-\frac {1}{2}N^{-1}(t)u .
\end{eqnarray*}
Clearly, $\cal H$ achieves the minimum value at $\gamma (D^*(t)k)$, i.e.
\begin{equation}\label{eq:4.2}
\begin{array}{ll} \displaystyle
{\cal H} (t, y,z,\gamma (D^*(t)k), k)= \min_{u \in U} {\cal H} (t,y,z,u,k).
\end{array}
\end{equation}

Under Assumptions \ref{ass:5.1} and  \ref{ass:5.2},
it is clear that Assumptions $\bf {(A.1)}$-$\bf {(A.4)}$ are satisfied.
Moreover, as $N$ is uniformly strictly positive-definite,
$N^{-1}$ is also strictly positive-definite and uniformly bounded. Then we have
\begin{eqnarray}
& \big ( D(t)\gamma(D^*(t)k_1)-D(t)\gamma (D^*(t)k_2), k_1-k_2 \big )_H =
-\frac {1}{2}(N^{-1}(t)D^*(t)(k_1-k_2), D^*(t)(k_1-k_2) \big )_H
<  0 , \\
& \| D(t)\gamma(D^*(t)k_1)-D(t)\gamma(D^*(t)k_2) \|_H
= \left \| -\frac {1}{2}D(t)N^{-1}(t)D^*(t)(k_1-k_2) \right \|_H \leq C ||k_1-k_2||_V.
\end{eqnarray}
Therefore, Assumption $\bf {(A.5)}$ is satisfied.

The stochastic Hamiltonian system of Problem \ref{pro:5.1} becomes
\begin{eqnarray}\label{eq:4.5}
\left\{
\begin{aligned}
dk(t)=& -\big[A^*(t)k(t)+2M(t)y(t)\big]dt-\big[B^*(t)k(t)+2Q(t)z(t)\big]dW(t),\\
dy(t)=& \ \big [A(t)y(t)+B(t)z(t)-\frac{1}{2}D(t)N^{-1}(t) D^*(t)k(t)+G(t)\big ]dt+z(t)dW(t),\\
k(0)=&-2hy(0), \quad y(T)=\xi,~~~~~~~t\in [0, T].
\end{aligned}
\right.
\end{eqnarray}

The next theorem gives the optimal solution to Problem \ref{pro:5.1}.

\begin{theorem}\label{thm:b2}
Let Assumptions \ref{ass:5.1} and  \ref{ass:5.2} be satisfied.
There exists a unique solution
$(k(\cdot), y(\cdot),z(\cdot))\in \mathbb {M}^2[0, T]$ of the Hamiltonian system \eqref{eq:4.5}
and Problem \ref{pro:5.1} has a unique optimal control
\begin{eqnarray}
u(t)=-\frac{1}{2}N^{-1}(t) D^*(t)k(t) .
\end{eqnarray}
\end{theorem}

\begin{proof}
Since Assumptions (\ref{ass:5.1})-(\ref{ass:5.2}) implies {\bf Assumptions (A)},
the following result is an immediate consequence of Theorem \ref{thm:4.3} and Theorem \ref{thm:4.4}.
\end{proof}

Having solved the linear-quadratic control problem formulated in the abstract evolution framework,
we now turn to  an optimal control of a Dirichlet problem for a linear backward stochastic parabolic PDE
and a quadratic cost functional. This problem is less abstract and serves as a more specific illustration
of our results.

We first state the problem in the specific (stochastic) PDE sense,
then reformulate it in our abstract framework using the stochastic evolution equation
and the Gelfand triple. Let us introduce some Sobolev spaces on a domain.
Let $\Lambda$ be a bounded, open set in ${\mathbb R}^d$ with boundary $\Gamma$,
which is $C^\infty$-manifold of dimension $d - 1$, and $L^2 (\Lambda)$ the set
of all square-integrable functions on $\Lambda$.
For $m = 0, 1$, we
define the space $H^m(\Lambda) \triangleq \{ \phi:
\partial_x^\alpha \phi (x)\in L^2 ( \Lambda ),  \ \mbox {for any}
\ \alpha: =( \alpha_1, \cdots, \alpha_d ) \ \mbox {with} \ |\alpha|
:= | \alpha_1 | + \cdots + | \alpha_d | \leq m \}$ with the following norm:
\begin{eqnarray*}
\| \phi \|_m \triangleq \left \{ \sum_{ |\alpha| \leq m } \int_{\Lambda} |
\partial_x^\alpha \phi (x) |^2 d z \right \}^{\frac{1}{2}} .
\end{eqnarray*}
The space $H^m(\Lambda)$ is a Sobolev space of order $m$ on $\Lambda$. For any
$u,v \in H^m(\Lambda)$, we define the the scalar product as
\begin{eqnarray} \label{eq:63}
(u,v)_{ H^m(\Lambda)} \triangleq  \sum_{ |\alpha| \leq m } \int_\Lambda
\partial_x^\alpha u (x) \partial_x^\alpha v (x) d x.
\end{eqnarray} It is
well-known that the space $H^m(\Lambda)$ endowed with the scalar product
\eqref{eq:63} is a Hilbert space. Define
\begin{eqnarray*}
H^1_0(\Lambda) \triangleq \{ \phi: \phi \in H^1(\Lambda), \phi\big|_{\partial \Lambda}=0\} .
\end{eqnarray*}
Denote by $H^{-1}(\Lambda)$ the dual space of $H^1_0(\Lambda).$  Then we see
$$H_0^{1}(\Lambda)\subset L^2(\Lambda)\subset H^{-1}(\Lambda)$$ is a Gelfand triple.

We consider the state $y (t,x) \in \mathbb R$ of a system at time
$t\in [0,T]$ and at the point $x \in {\bar \Lambda}= \Lambda\cup \partial \Lambda$,
which is given by the Dirichlet problem for the quasilinear backward stochastic
parabolic PDE:
 \begin{eqnarray}\label{eq:6.13}
  \left\{
  \begin{aligned}
      d y (t,x) = & \ \big\{-\partial_{x^i}\big[
      a^{ij}(t,x)\partial_{x^j}y(t,x) \big] -b^i(t,x)\partial_{x^i}y(t,x)-c(t,x)y(t,x)
      +\nu(t,x)z(t,x)+g(t,x)+ u(t,x) \big\} dt\\
      & +z(t,x)dW(t),\quad~~ (t,x)\in
    [0,T]\times \Lambda, \\
      y(T,x)=&~\xi(x),~~~~x\in \Lambda, \\
      y(t,x)=&~0,~~~~~~~ (t,x)\in
    [0,T]\times \partial \Lambda,
  \end{aligned}
  \right.
\end{eqnarray}
where $u(t,x)$ is the control process valued in $\mathbb R$. Here the coefficients $a^{ij}, b^i, c, \nu, g:
[ 0, T ]\times \Omega\times \Lambda\rightarrow \mathbb R$ and $\xi
: \Omega \times \Lambda\rightarrow \mathbb R$  are given measurable random mappings.
A control process $u(\cdot, \cdot)$ is said to be admissible if $u(\cdot, \cdot) \in {\cal M}^2_{\cal F} ( 0, T; L^2(\Lambda) )$.

For any admissible  control $u(\cdot, \cdot)$, the following
definition gives the generalized weak solution to Eq. \eqref{eq:6.13}

\begin{definition}
A pair of ${\mathscr P} \times {\mathscr B} (\Lambda
)$-measurable functions $(y (\cdot, \cdot), z(\cdot, \cdot))$ valued in $\mathbb R\times
\mathbb R$ is called a (generalized or weak) solution of
\eqref{eq:6.13}, if $y (\cdot, \cdot) \in {\cal M}_{\cal F}^2 ( 0, T;
H_0^1(\Lambda))$ and $z (\cdot, \cdot) \in {\cal M}_{\cal F}^2 ( 0, T;
 L^2(\Lambda)) $  such that for every $\phi \in H_0^1(\Lambda)$ and a.e. $( t,
\omega) \in [ 0, T ] \times \Omega $, it holds that
\begin{eqnarray}
\int_\Lambda y (t,x) \phi(x)dx  &=& \int_\Lambda\xi ( x )\phi(x)dx -
\int_t^T \int_\Lambda a^{ij} ( s, x )
\partial_{x^j} y ( s, x ) \partial_{x^i}\phi(x) dx d s
+\int_t^T \int_\Lambda \big [ b^i ( s, x ) \partial_{x^i} y ( s, x ) \\
&& + c ( s, x ) y ( s, x ) -\nu(s,x)z(s,x)
-g(s,x)- u(s,x) \big ] \phi(x) dx d s - \int_t^T \int_\Lambda  z(s,x) \phi(x) dxd W (s) . \nonumber
\end{eqnarray}
\end{definition}
For any admissible control process $u (\cdot, \cdot)$ and the solution $(y (\cdot, \cdot),z(\cdot, \cdot))$ of the
corresponding state equation \eqref{eq:6.13}, the objective of the control problem is to
minimize a quadratic cost functional as follows:
\begin{eqnarray}\label{eq:6.14}
&& \inf_{u(\cdot)\in  {\cal M}^2_{\cal F} ( 0, T; L^2(\Lambda) ) } \bigg \{
{\mbox E} \bigg [ \int_{\Lambda} y^2(0,x)dx \bigg ]
+ {\mbox E} \bigg [ \iint_{[0,T]\times{\Lambda}}y^2(s,x) dsdx \bigg ] \nonumber \\
&& \qquad\qquad\qquad\qquad\qquad + {\mbox E} \bigg [ \iint_{[0,T]\times{\Lambda}} z^2(s,x) dsdx \bigg ]
+ {\mbox E} \bigg [ \iint_{[0,T]\times{\Lambda}}u^2(s,x) dsdx \bigg ] \bigg \} .
\end{eqnarray}
To make the control problem well-defined, we now
fix some constants $K\in (1,\infty)$ and $\kappa\in (0,1)$ and give the following assumptions on coefficients:

\begin{assumption}\label{ass:6.3}
  The functions $a \triangleq \big ( a^{ij} \big )_{i, j = 1, 2, \cdots, d},
  b \triangleq \big ( b^i \big )_{i = 1, 2, \cdots, d}, c, \nu$ and $g$ are $\mathscr{P} \times
  \mathscr B(\Lambda)$-measurable with values in the set of real symmetric $d\times
  d$ matrices, $\mathbb{R}^{d}$, $\mathbb R, \mathbb{R}$ and $\mathbb{R}$, respectively and are bounded by $K$. The real
  function $\xi \in  L^2(\Omega,{\mathscr{F}}_T,P;L^2(\Lambda))$.
\end{assumption}

\begin{assumption}\label{ass:6.4}
  We assume that, for $a = \big ( a^{ij} \big )_{i, j = 1, 2, \cdots, d}$, the super-parabolic condition is satisfied, i.e.
  \begin{equation*}
    \kappa I
    \leq 2 a^{ij} (t,\omega,x) \leq K I, ~~~\forall~
    (t,\omega,x)\in [0,T]\times \Omega\times \mathbb R^{d}.
  \end{equation*}
\end{assumption}

To apply the abstract results in Theorem \ref{thm:b2}, we set
$V = H^1_0(G)$, $H = L^2(G)$, $V^* = H^{-1}(G)$, which form the Gelfand triple $(V, H, V^*)$.
We define the second-order differential operator $A$, the first-order differential operator $B$
and $G$ respectively by
\begin{eqnarray*}
A(t)\phi(x)\triangleq-\big\{\partial_{x^i}\big[
      a^{ij}(t,x)\partial_{x^j}\phi(x) \big] +b^i(t,x)\partial_{x^i}\phi(x)
       +c(t,x)\phi(x)\big\}, \quad \forall \phi \in V,
\end{eqnarray*}
\begin{eqnarray*}
B(t)\psi(x)\triangleq\nu(t,x)\psi(x), \ \forall \psi\in H, \quad \mbox{and} \quad G(t)(x)\triangleq g(t,x) .
\end{eqnarray*}
Note that the adjoint operator of $A$ reads
\begin{eqnarray*}
 A^*(t) \phi (x) \triangleq - \partial_{x^i} [ a^{ij}
( t, x ) \partial_{x^j} \phi (x) ] + b^i ( t, x ) \partial_{x^i}
\phi (x) - [c ( t, x )-\partial_{x^i}b^i(t,x)] \phi (x) , \quad \forall \phi \in V ,
\end{eqnarray*}
and the adjoint operator of $B$ is itself.
Now we can rewrite the state equation \eqref{eq:6.13} in the
following abstract backward stochastic evolution equation
in the Gelfand triple $(V, H, V^*)$:
\begin{eqnarray}\label{eq:10.1}
\left\{
\begin{aligned}
dy(t) =& \ [A(t)y(t)+B(t)z(t)+u(t)+G(t)]dt+ z(t)dW(t),\\
y(T) = & \ \xi.
\end{aligned}
\right.
\end{eqnarray}
The corresponding optimal control problem becomes
\begin{eqnarray}\label{eq:6.17}
\inf_{u(\cdot)\in  M^2_{\mathscr F}(0, T; U) }
\bigg \{  {\mbox E} \big [ ( y(0),y(0))_H \big ]
+ {\mbox E} \bigg [ \int_0^T(y(s),y(s))_H ds \bigg ]
+ {\mbox E} \bigg [ \int_0^T(z(s),z(s))_H ds \bigg ]
+ {\mbox E} \bigg [ \int_0^T (u(s),u(s))_H ds \bigg ] \bigg \}.
\end{eqnarray}
Thus, this optimal control problem is a special case of Problem \ref{pro:5.1},
in which the operators $D$, $M$, $Q$, $N$, $h$ are identity operators.
Under Assumptions \ref{ass:6.3}-\ref{ass:6.4},
it can be shown that the optimal control problem \eqref{eq:6.17} satisfies
Assumptions \ref{ass:5.1}-\ref{ass:5.2} or {\bf Assumption (A)}.
Consequently, we can apply Theorem \ref{thm:b2} to confirm that the stochastic
Hamiltonian system:
\begin{eqnarray}\label{eq:4.5}
\left\{
\begin{aligned}
dk(t)=& -\big[A^*(t)k(t)+2y(t)\big]dt-\big[B^*(t)k(t)+2z(t)\big]dW(t),\\
dy(t)=& \ \big [A(t)y(t)+B(t)z(t)-\frac{1}{2}k(t)+G(t)\big ]dt+z(t)dW(t),\\
k(0)=&-2y(0), \quad y(T)=\xi,~~~~~~~t\in [0, T].
\end{aligned}
\right.
\end{eqnarray}
has a unique solution $(k (\cdot), y (\cdot), z (\cdot)) \in {\mathbb M}^2 [0, T]$ and that the optimal control
is given by the following rule:
\begin{eqnarray}
u (t) = -\frac{1}{2}k(t).
\end{eqnarray}
Alternatively, the optimal control can be expressed by
\begin{eqnarray}
u(t,x) =-\frac{1}{2}k(t,x) ,
\end{eqnarray}
where $k (t, x)$ is the unique solution of the following stochastic PDE:
\begin{eqnarray}
\left\{
\begin{aligned}
dk(t,x)=& -\big[- \partial_{x^i} [ a^{ij}
( t, x ) \partial_{x^j} k (t,x) ] + b^i ( t, x ) \partial_{x^i}
k(t,x) - [c ( t, x )-\partial_{x^i}b^i(t,x)] k (t,x)+2y(t,x)\big]dt
\\&-\big[\nu(t,x)k(t,x)+2z(t,x)\big]dW(t),\\
k(0,x)=&-2y(0,x), ~~~~~~x\in \Lambda . \\
      k(t,x)=&~0,~~~~~~ (t,x)\in
    [0,T]\times \partial \Lambda.
\end{aligned}
\right.
\end{eqnarray}
This stochastic PDE is equivalent to the forward part of the stochastic Hamiltonian system (\ref{eq:4.5}).

\section{Conclusion}

In this paper, the existence of optimal controls is studied under infinite-dimensional stochastic backward systems.
The controlled BSPDEs are represented in the abstract evolution form, i.e. BSEEs. This allows us to show the existence
of optimal controls straightforward using the uniqueness and existence of a solution to FBSEE and the maximum
principle for the controlled BSEE. Two examples of infinite-dimensional linear-quadratic stochastic control problems
are solved to illustrate our results.

\end{document}